\documentclass[11pt]{elsarticle}
\usepackage{latexsym,amssymb,amsmath,graphics}
\usepackage{comment}

\def\H{{\mathcal H}}
\def\al{{\alpha}}

\DeclareMathOperator{\supp}{supp}
\DeclareMathOperator{\vect}{span}
\def\N{{\mathbb{N}}}

\def\R{{\mathbb{R}}}

\def\d{\,{\mathrm{d}}}
\newcommand{\norm}[1]{{\left\|{#1}\right\|}}
\newcommand{\ent}[1]{{\left[{#1}\right]}}

\DeclareMathOperator{\sinc}{sinc}
\DeclareMathOperator{\pv}{p.v.}
\newcommand{\va}[1]{\psi_{#1,c}}
\newcommand{\ja}[1]{j_{#1,c}}
\newcommand{\jc}[1]{{k}^{\alpha}_{#1,c}}

\newtheorem{lemma}{Lemma}[section]
\newtheorem{proposition}[lemma]{Proposition}
\newtheorem{theorem}[lemma]{Theorem}
\newtheorem{corollary}[lemma]{Corollary}

\newdefinition{definition}[lemma]{Definition}
\newdefinition{example}[lemma]{Example}

\newdefinition{remark}[lemma]{Remark}
\newdefinition{notation}[lemma]{Notation}
\newproof{proof}{Proof}
\newproof{proofth}{Proof of Theorem~\ref{thm:extension}}

\begin{document}
\title{Almost Everywhere Convergence of Prolate Spheroidal Series}

\author{Philippe Jaming\corref{cor1}}
\ead{philippe.jaming@u-bordeaux.fr}
\author{Michael Speckbacher\fnref{fn2}}
\address{Institut de Math{\'e}matiques de Bordeaux, Universit{\'e} de Bordeaux \\351, cours de la Lib{\'e}ration, 33405 TALENCE \\ France }
\ead{speckbacher@kfs.oeaw.ac.at}

\cortext[cor1]{Corresponding author}
\fntext[fn2]{ M.S. was supported by the Austrian Science Fund (FWF) through an Erwin-Schr{\"{o}}dinger Fellowship (J-4254).}

\begin{abstract}
\noindent In this paper, we show that the expansions of functions from $L^p$-Paley-Wiener type spaces in terms of the prolate spheroidal wave functions converge almost everywhere for $1<p<\infty$, even in the cases when they might not converge in $L^p$-norm. We thereby consider the classical Paley-Wiener spaces $PW_c^p\subset L^p(\R)$ of functions whose Fourier transform is supported in $[-c,c]$ and Paley-Wiener like spaces $B_{\al,c}^p\subset L^p(0,\infty)$ of functions whose Hankel transform $\H^\al$ is supported in $[0,c]$.
As a side product, we show the continuity of the projection operator $P_c^\al f:=\H^\al(\chi_{[0,c]}\cdot \H^\al f)$ from $L^p(0,\infty)$ to $L^q(0,\infty)$, $1<p\leq q<\infty$.
\end{abstract}

\maketitle

\section{Introduction}

The prolate spheroidal wave functions (PSWF) form an orthonormal basis of $L^2(\R)$ that is best concentrated in the time-frequency plane. They were first studied in  the seminal work of Landau, Pollak, and Slepian \cite{lapo61,lapo62,slepo61,sle83} at Bell Labs which inspired extensive research  due to their efficiency as a tool for signal processing.
They have since found further applications
in various fields such as random matrix theory ({\it e.g.} \cite{dy62,me04}), spectral estimation \cite{tho82,abro17}
or numerical analysis ({\it e.g.} \cite{XRY, Wa2}).

While the natural setting for prolate spheroidal wave functions is $L^2(\R)$, the question  of convergence of their series expansions arises in $L^p(\R)$. This issue has been solved by Barcel\'o and Cordoba \cite{baco89} who showed
that the expansion of an $L^p$-band limited function $f$ in the PSWF-basis converges to $f$ whenever $4/3<p<4$,
and that this range of $p$'s is optimal. That result was recently extended in \cite{bojaso19} to several natural variants of the prolates like the Hankel prolates.

The aim of this paper is to continue this work by investigating almost-everywhere convergence properties of PSWFs expansions of functions in
$L^p(\R)$. This question is very natural in view of Carleson's fundamental result about Fourier series
\cite{carl66} which was extended to several orthonormal bases of polynomials, {\it e.g.} by
Pollard for Legendre series \cite{po72}. Their almost everywhere convergence results follow from very delicate analysis of the maximal operator associated to the expansions considered. On the other hand, the situation
is much better for expansions in terms of spherical Bessel functions ({\it see} \cite{cigupeva99,civa09}) where almost-everywhere convergence is surprisingly simple once mean convergence has been established. The key here is the fast decay of Bessel functions with respect to its parameter.
Our main result is to show that the situation is similar for PSWF series since the PSWF-basis can be nicely expressed in terms of spherical Bessel functions.

\smallskip

Let us now be more precise and introduce some notation before giving the exact statements.
 For $f\in L^1(\R)$, the Fourier transform  is given by 
$$
\widehat{f}(\xi)=\mathcal{F}(f)(\xi)=\int_\R f(t)e^{-2\pi i \xi t}\d t,
$$
and the definition  extends to $L^2(\R)$ and $\mathcal{S}'(\R)$ (the space of tempered distributions) in the usual way.
The Paley-Wiener spaces of band-limited functions are denoted by
$$
PW_c^{p} =\left\{f\in L^p(\R):\ \supp \big(\widehat f\hspace{0.1cm} \big)\subseteq[-c,c]\right\},
$$
where the Fourier transform is to be understood in the distributional sense whenever $p>2$. The projection 
$P_cf=\mathcal{F}^{-1}\big(\chi_{[-c,c]}\cdot\mathcal{F}(f)\big)$ is a continuous operator from $L^p(\R)$ to $PW_c^p$, $1<p<\infty$.

The Prolate Spheroidal Wave Functions (PSWFs) are eigenfunctions of an integral operator and, using the min-max theorem, can be defined by the following extremal problem
$$
\psi_{n,c}:=\mathrm{argmax}\left\{\frac{\norm{f}_{L^2(-1,1)}}{\norm{f}_{L^2(\R)}}\,: f\in PW_c^2,\ 
f\in\vect\{\psi_{k,c}:\ k<n\}^\perp\right\}.
$$
The family $(\psi_{n,c})_{n\geq 0}$ forms an orthonormal basis for $PW_c^2$ and    satisfies also
$$
\int_{-1}^1\psi_{n,c}(t)\overline{\psi_{m,c}(t)}dt=\lambda_n \delta_{n,m} ,
$$
which is often referred to as double orthogonality.

The central object that we study in this paper is given by 
$$
\Psi_N f:=\sum_{n=0}^N \langle f,\psi_{n,c}\rangle \psi_{n,c}.
$$
It was shown in 
\cite{baco89} that $\Psi_Nf\rightarrow f$, $N\rightarrow\infty$, in $L^p(\R)$-norm for every $f\in PW_c^p$ if and only if $\frac{4}{3}<p<4$. Our main contribution in this paper is that the series $\Psi_Nf$ converges also almost everywhere, and that the range of convergence extends to $1\leq p<\infty$. Moreover, for general functions $f\in L^p(\R)$, $\Psi_N f$ converges almost everywhere to $P_cf$.
\begin{theorem}\label{thm:extension}
If $1< p<\infty$, and $f\in L^p(\R)$, then $\Psi_Nf\rightarrow P_cf$ almost everywhere. For $p=1$, we have that  $\Psi_Nf\rightarrow f$ almost everywhere for every $f\in PW^1_c$.
\end{theorem}

\noindent For  $f\in L^1(0,\infty)$, the Hankel transform is given by 
$$
\mathcal H^{\alpha}f(x):=\int_0^{\infty}f(y)\sqrt{xy}J_{\alpha}(xy) \d y,
$$
where 
$J_\alpha$ denotes the Bessel function of first kind and order $\alpha>-\frac{1}{2}$. Like the Fourier transform, the 
Hankel transform extends to a unitary operator on $L^2(0,\infty)$ and in a distributional sense to $L^p(0,\infty)$, $p>2$. Similar to the Fourier transform, a function cannot be confined to a finite interval in both time and Hankel domain. See \cite{bow71,ghoja11,rovo99}, for further uncertainty principles for the Hankel transform. 

The Paley-Wiener type  spaces $B_{\al,c}^p$ are defined as
$$
B_{\alpha,c}^p :=\left\{f\in L^p(0,\infty):\ \supp \big(\mathcal{H}^{\alpha}(f)\big)\subseteq[0,c]\right\},
$$
and the band-limiting projection operator is given by $P^\al_c f=\H^\al(\chi_{[0,c]}\cdot \H^\al f)$. We will show in Theorem~\ref{thm:projection-cont} that this operator is bounded on $L^p(0,\infty)$, for $1<p<\infty$.
Finally, the Circular (Hankel) Prolate Spheroidal Wave Functions (CPSWFs) were first introduced and studied in \cite{sle64}. They are defined by
$$
\varphi_{n,c}^\alpha:=\mathrm{argmax}\left\{\frac{\norm{f}_{L^2(0,1)}}{\norm{f}_{L^2(0,\infty)}}\,: f\in B_{\alpha,c}^2 ,\ 
f\in\vect\{\varphi_{k,c}^\alpha:\ k<n\}^\perp\right\}.
$$
The family $(\varphi_{n,c}^\alpha)_{n\geq 0}$ forms an orthonormal basis of $B_{\alpha,c}^2 $.
Note also that when $\alpha=1/2$, these are usual PSFWs, more precisely, $\varphi_{n,c}^{1/2}=\psi_{2n,c}$. Again we are interested in expansions of $f\in B_{\al,c}^p$ with respect to the CPSWFs, {\it i.e.} in the convergence of 
$$
\Phi_N f:=\sum_{n=0}^N \langle f,\varphi_{n,c}^\al\rangle \varphi_{n,c}^\al.
$$
It is shown in \cite[Theorem~5.6]{bojaso19}
that $\Phi_N f\rightarrow f $ in $L^p(0,\infty)$-norm for every $f\in B_{\al,c}^p$ if and only if $\frac{3}{4}<p<4$.  Like for the classical Paley-Wiener space, the series converges also almost everywhere and the range of convergence extends to $1\leq p<\infty$ for functions in $B_{\al,c}^p$.  

\begin{theorem}\label{thm:extension2}
Let $\alpha>-\frac{1}{2}$. If ${1< p<\infty}$, then $\Phi_Nf\rightarrow P^\al_c f$ almost everywhere for every $f\in L^p(0,\infty)$. Moreover, if  $p=1$, then $\Phi_Nf\rightarrow  f$ almost everywhere for every $f\in B^p_{\al,c}$.
\end{theorem}

This paper is organized as follows. In Section~\ref{sec:prel}, we collect some properties on spherical Bessel functions and the Muckenhoupt class that we will need for the proof of our main results. We then prove Theorem~\ref{thm:extension} in Section~\ref{sec:main}, and conclude with illustrating how this proof has to be adapted to show Theorem~\ref{thm:extension2} in Section~\ref{sec:circular}.

\medskip

As usual, we will write $A\lesssim B$ if there is a constant $C$ such that $A\leq CB$ and $A\sim B$ if
$A\lesssim B$ and $B\lesssim A$.

\section{Preliminaries}\label{sec:prel}
\subsection{Properties of spherical Bessel functions}\label{sec:Bessel}
\noindent In this section we   define two families of spherical Bessel functions tailored to form  orthonormal bases for the spaces $PW_c^2$ and $B_{\alpha,c}^2$ respectively.

For $\alpha>-1/2$, the Bessel functions of the first kind can be defined through the Poisson representation, {\it see e.g.} \cite[10.9.4]{NIST10}, 
$$
J_\alpha(x)=\frac{x^\alpha}{\pi^{1/2}2^\alpha\Gamma(\alpha+\frac{1}{2})}\int_0^1(1-t^2)^{\alpha-\frac{1}{2}}\cos(xt)\d t.
$$
Recall that $J_\alpha$ satisfies the pointwise bound
\begin{equation}\label{eq:est:T_a}
|J_\alpha(x)|\lesssim \frac{|x|^{\alpha+1/2}}{2^\alpha\Gamma(\alpha+\frac{1}{2})}.
\end{equation}
It is shown in \cite{baco89} and \cite{bojaso19} that for $1<p<\infty$
\begin{equation}
	\label{eq:lpnorm}
\norm{x^{-1/2}J_{n+1/2}}_{L^p(\R)} 
\sim \begin{cases}
		n^{-1+\frac{1}{p}}&\mbox{when }1<p<4\\
		n^{-\frac{3}{4}}\log n&\mbox{when }p=4\\
		n^{-\frac{5}{6}+\frac{1}{3p}}&\mbox{when }p>4
	\end{cases}.
\end{equation}
Note that the nature of the right hand side shows that $\norm{x^{-1/2}J_{2n+\alpha+1}}_{L^p(0,\infty)}$
satisfies the same bounds though with different constants.

The classical (dilated) spherical Bessel functions are defined by
\begin{equation}
\label{eq:sphbess1}
\ja{n}(x):=\sqrt{\frac{2n+1}{2}}\frac{J_{n+1/2}(cx)}{\sqrt{x}}.
\end{equation}
They  satisfy the orthogonality relations
$$
\int_{\R}j_{n,c}(x)j_{m,c}(x)\d x=\delta_{n,m},
$$
and their Fourier transforms are given by 
$$
\widehat{j_{n,c}}(\xi)=(-1)^n \sqrt{\frac{2n+1}{\pi c}}P_n\left(\frac{\xi}{c}\right)\cdot\chi_{[-c,c]}(\xi),
$$
where $P_n$ denotes the Legendre polynomial of degree $n$. For $p>1$, one has that $j_{n,c}\in L^p(\R)$ and consequently $j_{n,c}\in PW_c^p$.
Moreover, by \eqref{eq:lpnorm} there exists $\gamma_p<\frac{1}{2}$ such that $\|j_{n,c}\|_{L^p(\R)}\lesssim n^{\gamma_p}$, for every $1<p<\infty$, and \eqref{eq:est:T_a} implies that, for $x$ fixed,
\begin{equation}\label{eq:pointwise-bessel}
|\ja{n}(x)|\lesssim \frac{\sqrt{2n+1}}{n!} \left(\frac{c|x|}{2}\right)^n\lesssim n^{-2}.
\end{equation}
The second family of spherical Bessel functions given by 
\begin{equation}
\label{eq:sphbess2}
\jc{n}(x):=\sqrt{2(2n+\alpha+1)}\frac{J_{2n+\alpha+1}(cx)}{\sqrt{x}},
\end{equation}
obeys the orthogonality relation 
$$
\int_{0}^{\infty}\jc{n}(x)\jc{m}(x)\d x=\delta_{n,m}.
$$
Their Hankel transforms are given by 
\begin{equation}
\label{eq:besseljacobi}
\mathcal H^{\alpha}(\jc{n})(x)=\sqrt{\frac{2(2n+\alpha+1)}{c}}\left(\frac{x}{c}\right)^{\alpha+\frac{1}{2}}
P_n^{(\alpha,0)}\left(1-2\left(\frac{x}{c}\right)^2\right)
\chi_{[0,c]}(x),
\end{equation}
 where $P_n^{(\alpha,0)}$ denotes the Jacobi polynomials of degree $n$ and
parameter $\alpha$, normalized so that
$P_n^{(\alpha,0)}(1)=\frac{\Gamma(n+\alpha+1)}{\Gamma(n+1)\Gamma(\alpha+1)}$, see for example \cite{Slepian3}.
As before,  for $1<p<\infty$,   the $L^p(0,\infty)$ norms of $\jc{n}$ are bounded like
$	\norm{\jc{n}}_{L^p(0,\infty)} \lesssim n^{\gamma_p}$ for some $0\leq \gamma_p<1$. Hence, $\jc{n}\in B_{\alpha,c}^p$. By \eqref{eq:est:T_a} we also get that
\begin{equation}\label{eq:pointwise-bessel2}
|\jc{n}(x)|\lesssim \frac{\sqrt{2n+\alpha+1}}{\Gamma(2n+\alpha+1)} \left(\frac{c|x|}{2}\right)^{2n}\lesssim n^{-2},
\end{equation}
where the last estimate holds for $x$ fixed with a constant that depends on $x$ (which can be chosen uniformly if $x$ is in a compact set).

\subsection{The Hilbert Transform on Weighted ${L^p}$-spaces}\label{sec:hilbert-trafo}

Let $J\subset \R$ be an interval. The class of  Muckenhoupt weights $A^p(J)$, $1<p<\infty$, consists of all functions $\omega\,: J\to\R_+$  such that 
$$
\ent{\omega}_{p}:=
\sup_K\left(\frac{1}{|K|}\int_K\omega(x)\,\mbox{d}x\right)\left(\frac{1}{|K|}\int_K\omega(x)^{-\frac{q}{p}}\,\mbox{d}x\right)^{\frac{p}{q}}<\infty,
$$
where the supremum is taken over all finite length intervals $K\subset J$.

The Hilbert transform is defined as
$$
H f(x)=\frac{1}{\pi}\int_J\frac{f(y)}{x-y}\,\mathrm{d}y,
$$
where the integral has to be taken in the principal value sense.

Hunt, Muckenhaupt and Wheeden \cite{humuwhe73} showed that the Hilbert transform extends to a bounded linear operator on
$L^p(J,\omega)\to L^p(J,\omega)$ if and only if $\omega$ is an $A^p(J)$ weight, and the sharp dependence of the operator norm on $\ent{\omega}_{p}$ was established by Petermichl \cite{petermichl}. Let us also recall the well known fact that $x^\beta\in A^p(0,\infty)$ if and only if $-1<\beta<p-1$.

\section{Proof of Theorem~\ref{thm:extension}}\label{sec:main}

\noindent The prolate spheroidal wave functions $(\va{n})_{n\geq 0}$ can be expressed in terms of the spherical Bessel functions $(\ja{n})_{n\geq 0}$ as
$$
\va{n}=\sum_{k\geq 0}b_k^n\ \ja{k}.
$$
It follows from  \cite[Eqs. (8) \& (9)]{baco89} that there exists $n_0\in\N$ such that for all $n\geq n_0$, the numbers $b_k^n$ satisfy
\begin{enumerate}[(i)]
\item $|b_0^n|\lesssim  n^{-2}$,\label{enum:1}
\item for $k\geq 1$, $|b_k^n|\lesssim n^{-|k-n|}$.\label{enum:2}
\end{enumerate}
Note that indexes of the base change coefficients in \cite{baco89} are shifted, and that these estimates can  also be obtained from \cite{bojaso19}.
It was proven in  \cite[Lemma~2.6]{bojaso19} that if the above conditions are satisfied, then  $\|\va{n}\|_{L^p(\R)}\lesssim n^{\gamma_p}$.

\begin{lemma}\label{lem:pointwise}
If $1<p<\infty$, then 
$\Psi_N f$ converges absolutely in every point (and uniformly on compact sets) to some function $h$ for every $f\in L^p(\R)$. 
\end{lemma}
\begin{proof}
Let us first rewrite $\va{n}$ as  
\begin{align*}
\va{n}&=\sum_{k\geq 0}b_k^n \ja{k}\\&=b_0^n\ja{0}+b_n^n\ja{n}+b_{n-1}^n\ja{n-1}+b_{n+1}^n\ja{n+1}+\sum_{k=1}^{n-2}b_k^n\ja{k}+\sum_{k\geq n+2}b_k^n\ja{k}.
\end{align*}
Subsequently,   we  estimate each of the summands separately.  First, by \eqref{enum:1} we have $|b_0^n\ja{0}(x)|\lesssim  n^{-2}$, and  $|b_n^n\ja{n}(x)|\leq|\ja{n}(x)|\lesssim n^{-2}$ by \eqref{eq:pointwise-bessel}. If $n\geq 2$, the last two terms may be bounded using \eqref{enum:2} and \eqref{eq:pointwise-bessel} as
$$
\sum_{k\geq n+2}|b_k^n\ja{k}(x)|\lesssim \sum_{k\geq n+2} n^{-|k-n|}k^{-2}\leq n^{-2}\sum_{k\geq 0} n^{-(k+2)}\lesssim n^{-2},
$$
and
$$
\sum_{k=1}^{n-2}|b_k^n\ja{k}(x)|\lesssim  \sum_{k=1}^{n-2} n^{-|k-n|}k^{-2}\leq \sum_{k=2}^{n-1} n^{-k}=n^{-2}\sum_{k=0}^{n-3} n^{-k}\lesssim n^{-2}.
$$
The third and fourth term may also easily be bounded using \eqref{enum:2}: 
$$
|b_{n-1}^n\ja{n-1}(x)|\lesssim n^{-1}(n-1)^{-2}\lesssim n^{-2},
$$
and
$$
|b_{n+1}^n\ja{n+1}(x)|\lesssim n^{-1}(n+1)^{-2}\lesssim n^{-2}.
$$
Therefore, $|\va{n}(x)|\lesssim n^{-2}$ which implies by  H{\"o}lder's inequality that
\begin{align*}
 \sum_{n=0}^N|\langle f,\va{n}\rangle \va{n}(x)|&\leq \sum_{n=0}^N\| f\|_{L^p(\R)}\|\va{n}\|_{L^q(\R)} |\psi_n(x)|\\ &\lesssim \| f\|_{L^p(\R)}\sum_{n=0}^N n^{\gamma_q-2}<\infty.
\end{align*}
\hfill\qed
\end{proof}

\noindent The following lemma is  well-known to the majority of our readers. 
We will nevertheless include a proof here. 

\begin{lemma}\label{lem:pw-dense}
If $1\leq p\leq q<\infty$, then $PW_c^p\hookrightarrow PW_c^q$ where the inclusion map is  continuous and dense. Moreover, the projection $P_c$ maps $L^p(\R)$ boundedly into  $PW_c^p$ if $1<p<\infty$.
\end{lemma}

\begin{proof}
Let $\vartheta_1,\vartheta_2$ be two functions from the Schwartz space such that $\text{supp}(\widehat{\vartheta}_1)\subseteq [-c,c]$, $\widehat{\vartheta}_2(\xi)=1,$ for every $\xi\in[-c,c]$, and $\|\vartheta_1-\vartheta_2\|_{L^1(\R)}\leq\varepsilon$. If $f\in PW_c^p$, then $f\ast\vartheta_2=f$ and, by Young's convolution inequality, one has $$\|f\|_{L^q(\R)}=\|f\ast\vartheta_2\|_{L^q(\R)}\leq \|f\|_{L^p(\R)}\|\vartheta_2\|_{L^\frac{pq}{pq+p-q}(\R)}.$$ This  shows that $PW_c^p$ is continuously embedded into $PW_c^q$. 

Let us now show the density. 
For every $\varepsilon>0$, and $f\in PW_c^q$ we may choose  $f_\varepsilon\in L^p(\R)\cap L^q(\R)$,  such that $\|f- f_\varepsilon\|_{L^q(\R)}=\varepsilon$. Again by Young's convolution inequality   we thus obtain
\begin{align*}
\|  f_\varepsilon\ast\vartheta_1-f\|_{L^q(\R)}&\leq \| f_\varepsilon\ast\vartheta_1-f\ast\vartheta_1\|_{L^q(\R)}+\|f\ast\vartheta_1-f\ast \vartheta_2\|_{L^q(\R)}\\
&\leq \|f_\varepsilon-f\|_{L^q(\R)}\|\vartheta_1\|_{L^1(\R)}+\|f\|_{L^q(\R)}\|\vartheta_1-\vartheta_2\|_{L^1(\R)}\lesssim\varepsilon.
\end{align*}
Finally, $f_\varepsilon\ast \vartheta_1\in PW_c^p$ as $\text{supp}\big(\mathcal{F}(f_\varepsilon\ast \vartheta_1)\big)\subseteq \text{supp}(\widehat \vartheta_1)$, and $\|f_\varepsilon\ast \vartheta_1\|_{L^p(\R)}\leq\|f_\varepsilon\|_{L^p(\R)}\|\vartheta_1\|_{L^1(\R)}$. 

Concerning the continuity of $P_c$, we observe that for almost every point $x\in\R$, $P_c f(x)$ may be written as the sum of Hilbert transforms
\begin{align*}
P_c f(x)&=f\ast 2c\sinc 
(2c\hspace{0.08cm} \cdot\hspace{0.06cm})(x)=\int_\R f(t)\frac{\sin(2\pi c(x-t))}{\pi(x-t)}dt\\
&=\sin(2\pi cx)\pv\int_\R f(t)\frac{\cos(2\pi ct)}{\pi(x-t)}dt\\
&\quad-\cos(2\pi c x)\pv\int_\R f(t)\frac{\sin(2\pi ct)}{\pi(x-t)}dt\\
& =\sin(2\pi cx)H\big(f\cdot\cos(2\pi c\hspace{0.08cm}\cdot\hspace{0.06cm})\big)(x)-\cos(2\pi c x)H\big(f\cdot\sin(2\pi c\hspace{0.08cm}\cdot\hspace{0.06cm})\big)(x).
\end{align*}
By the continuity of the Hilbert transform on $L^p(\R)$ in the range $1<p<\infty$, it thus follows that
\begin{align*}
\|P_c f\|_{L^p(\R)}\hspace{-1pt}\leq\hspace{-1pt} \|H\big(f\cdot\cos(2\pi c\hspace{0.06cm}\cdot\hspace{0.06cm})\big)\|_{L^p(\R)}\hspace{-1pt}+\hspace{-1pt}\|H\big(f\cdot\sin(2\pi c\hspace{0.06cm}\cdot\hspace{0.06cm})\big)\|_{L^p(\R)}\hspace{-1pt}
\lesssim\hspace{-1pt}\|f\|_{L^p(\R)}.
\end{align*}\hfill\qed
\end{proof}

\begin{proposition}\label{prop:pointwise-ae}
If  $\frac{4}{3}<p<4$, then $\Psi_N f\rightarrow P_cf$ almost everywhere for every $f\in L^p(\R)$. 
\end{proposition}

\begin{proof}
First, as $P_c$ is continuous for $\frac{4}{3}<p<4$, it follows that $\langle f,\ja{n}\rangle=\langle f,P_c\ja{n}\rangle=\langle P_cf,\ja{n}\rangle$ and therefore $\Psi_N f=\Psi_N P_cf$.
Since $\Psi_N P_cf\rightarrow P_cf$ in $L^p(\R)$ by \cite[Theorem 1]{baco89}, there exists a subsequence $(\Psi_{N_k}P_cf)_{k\geq0}$ that converges almost everywhere to $P_cf$. From Lemma~\ref{lem:pointwise} we know that $\Psi_{N}f$ converges pointwise for every $x\in \R$ to a function $h$. Consequently, $P_c f=h$ almost everywhere. \hfill \qed
\end{proof}

\noindent We have now gathered all ingredients to prove our main theorem.

\begin{proofth}
If $1<p\leq 2$, then $P_c f\in PW_c^2$ for every $f\in L^p(\R)$  by Lemma~\ref{lem:pw-dense}. In this case, the result follows from Proposition~\ref{prop:pointwise-ae} and the identity $\Psi_N f=\Psi_N P_cf$. Similarly, the statement for $p=1$ holds as $PW_c^1\subset PW_c^2$. If $2<p<\infty$, and $f\in L^p(\R)$, then there exists a sequence $(f_k)_{k\geq0}\subset PW_c^2$  by Lemma~\ref{lem:pw-dense}  such that $f_k\rightarrow  P_c f\in PW_c^p$ and $\|f_k\|_{L^p(\R)}\lesssim\| f\|_{L^p(\R)}$. Without loss of generality, we may assume that $f_k\rightarrow P_cf$ almost everywhere and define 
$X_f:=\{x\in\R:\ \lim_{k\rightarrow\infty}f_k(x)= P_c f(x)\}$, as well as 
$$
X_k:=\Big\{x\in\R:\ \lim_{N\rightarrow\infty}\Psi_Nf_k(x)=f_k(x)\Big\}.
$$
 From Proposition~\ref{prop:pointwise-ae} we know that each $\R\backslash X_k$ has Lebesgue measure zero and, consequently, so has $\R\backslash\big(\bigcap_{k\geq 0}X_k\cap X_f\big)=\bigcup_{k\geq 0}(\R\backslash X_k)\cup (\R\backslash X_f)$. 

Using \eqref{eq:pointwise-bessel} and the estimate on the $L^p$-norms of $\ja{n}$ we derive
$$
\sum_{n\geq 0}|\langle f_k,\psi_n\rangle \psi_n(x)|\lesssim \|f\|_{L^p(\R)}\sum_{n\geq 0}n^{\gamma_q-2}<\infty.
$$
Now, if
$x\in \bigcap_{k\geq 0}X_k\cap X_f$, we may conclude
by the dominated convergence theorem that 
\begin{align*}
P_cf(x)&=\lim_{k\rightarrow\infty}f_k(x)=\lim_{k\rightarrow\infty}\lim_{N\rightarrow\infty}\Psi_Nf_k(x)=\lim_{N\rightarrow\infty}\lim_{k\rightarrow\infty}\Psi_Nf_k(x)
\\
&=\lim_{N\rightarrow\infty}\Psi_N P_cf(x)=\lim_{N\rightarrow\infty}\Psi_N f(x).
\end{align*}
\hfill\qed
\end{proofth}


\section{Adaptations for Circular PSWF }\label{sec:circular}

\noindent By  a combination of several lemmas in \cite{bojaso19} (in particular Lemmas~2.5, 5.3 \& 5.4) one can deduce that the change of base  coefficients $ b_m^n$ in 
$$
\varphi_{n,c}^\alpha=\sum_{m\geq 0}b_m^n\ \jc{m}
$$
satisfy the conditions \eqref{enum:1} and \eqref{enum:2}. Hence,
  $\|\varphi_{n,c}^\alpha\|_{L^p(0,\infty)}\lesssim n^{\gamma_p}$, for some $0\leq \gamma_p<1$.
The proof of Lemma~\ref{lem:pointwise} can therefore be transferred one-to-one   to show that 
$\Phi_Nf$ converges pointwise to a function $h$ for every $f\in L^p(0,\infty)$, $1< p<\infty$. 

The analogue of Lemma~\ref{lem:pw-dense} needs some preparation. In particular, we need to show the continuity of the projection operator $P_c^\al$. In \cite{va94} the author shows that for the following convention of the Hankel transform  
$$
\mathcal{F}^\al f(x):=\frac{x^{-\al/2}}{2}\int_0^\infty f(t)J_\al \big(\sqrt{xt}\big)t^{\al/2}dt,
$$
the projection  $\mathcal{F}^\al(\chi_{[0,c]}\cdot\mathcal{F}^\al)(x)$ acts as a continuous operator on the weighted Lebesgue spaces $L^p((0,\infty),x^\al)$ whenever $4\frac{\al+1}{2\al+3}<p<4\frac{\al+1}{2\al+1}$. 
Although different definitions of the Hankel transform can be related to each other and their $L^2$-theory coincides, this is no longer true for general $L^p$-spaces,  $p\neq 2$. 

For instance, when considering the $L^p\big((0,\infty),x^\al\big)$-convergence of spherical Bessel expansions, Varona \cite{va94}  showed that these series converge in $L^p$ if and only if $\max\Big\{\frac{4}{3},4\frac{\al+1}{2\al+3}\Big\}<p<\min\Big\{4,4\frac{\al+1}{2\al+1}\Big\}$. The restriction with respect to $\alpha$ however disappears in  \cite{bojaso19} due to a different convention for the Hankel transform and $L^p(0,\infty)$-convergence holds if and only if $\frac{4}{3}<p<4$.

As a second example, we show in the following that for the convention chosen in this paper, the projection $P_c^\al$ is continuous for $1<p<\infty$.
We will follow the arguments of the proof of \cite[Theorem 2]{va94}. 
\begin{theorem}\label{thm:projection-cont}
If $\al>-\frac{1}{2}$, and $1<p\leq q<\infty$, then 
$$
\|P^\al_c f\|_{L^q(0,\infty)}\lesssim \|f\|_{L^p(0,\infty)},\qquad f\in L^p(0,\infty).
$$
\end{theorem}

\begin{proof} Let us first write down $P^\al_c$ explicitly
\begin{align*}
P^\al_c f(x)&=\H^\al \big(\chi_{[0,c]}\cdot\H^\al f\big)(x)=\int_0^c\H^\al f(t)\sqrt{tx}J_\al(tx)\d t
\\
&= \int_0^c  \int_0^\infty f(y)\sqrt{ty}J_\al(ty)dy\sqrt{tx}J_\al(tx)\d t
\\
&=   \int_0^\infty f(y)\sqrt{xy}\int_0^cJ_\al(tx)J_\al(ty)t\d t\d y
\\
&=   \int_0^\infty f(y)\sqrt{xy} c\frac{yJ_{\al-1}(cy)J_\al (cx)-xJ_{\al-1}(cx)J_\al(cy)}{x^2-y^2}\d y,
\end{align*}
where we have used the explicit expression for Lommel's integrals \cite[p. 101]{bow58}  to derive the last equality.
By a change of variables one then obtains
\begin{align*}
&P^\al_c f(\sqrt{x})\\
&=   \int_0^\infty f(\sqrt{y})(xy)^{1/4} c\frac{\sqrt{y}J_{\al-1}(c\sqrt{y})J_\al (c\sqrt{x})-\sqrt{x}J_{\al-1}(c\sqrt{x})J_\al(c\sqrt{y})}{x-y}
\frac{\d y}{2y^{1/2}}\\
&=W_1f(x)-W_2f(x),
\end{align*}
where 
$$
W_1f(x):= \frac{\pi cx^{1/4}J_\al(c\sqrt{x})}{2}H\Big(f(\sqrt{y})y^{1/4}J_{\al-1}(c\sqrt{y})\Big)(x),
$$
and
$$
W_2f(x):=\frac{\pi cx^{3/4}J_{\al-1}(c\sqrt{x})}{2}H\Big(f(\sqrt{y})y^{-1/4}J_{\al}(c\sqrt{y})\Big)(x).
$$
Note that both $x^{-1/2}$, and $x^{p/2-1/2}$ are Muckenhoupt $A^p(0,\infty)$ weights by the remark in Section~\ref{sec:hilbert-trafo}. Hence, the Hilbert transform is continuous on the respective weighted Lebesgue spaces for every $1<p<\infty$. As $x^{1/4}J_\al(c\sqrt{x})\lesssim 1$, we thus have
\begin{align*}
\int_0^\infty |W_1f(x)|^px^{-1/2}\d x&\lesssim\int_0^\infty \left|H\Big(f(\sqrt{y})y^{1/4}J_{\al-1}(c\sqrt{y})\Big)(x)\right|^px^{-1/2}\d x
\\ 
&\lesssim \int_0^\infty \left|f(\sqrt{y})y^{1/4}J_{\al-1}(c\sqrt{y})\right|^py^{-1/2}\d y
\\
&\lesssim \int_0^\infty \left|f(\sqrt{y}) \right|^py^{-1/2}\d y= 2\int_0^\infty \left|f( y) \right|^p\d y,
\end{align*}
as well as 
\begin{align*}
\int_0^\infty |W_2f(x)|^px^{-1/2}\d x&\lesssim \int_0^\infty \left|H\Big(f(\sqrt{y})y^{-1/4}J_{\al}(c\sqrt{y})\Big)(x)\right|^px^{p/2-1/2}\d x
\\ 
&\lesssim\int_0^\infty \left|f(\sqrt{y})y^{-1/4}J_{\al}(c\sqrt{y})\right|^py^{p/2-1/2}\d y
\\
&\lesssim \int_0^\infty \left|f(\sqrt{y}) \right|^py^{-1/2}\d y= 2\int_0^\infty \left|f( y) \right|^p\d y.
\end{align*}
This  proves the case $p=q$ using $\|P^\al_c f\|_{L^p(0,\infty)}\leq \|W_1 f(x^2)\|_{L^p(0,\infty)}+\|W_2f(x^2)\|_{L^p(0,\infty)}$ and a change of variables.

Now, notice that for every $1\leq p\leq \infty$ one has
$$
\|P^\al_c f\|_{L^\infty(0,\infty)}\lesssim\|\chi_{[0,c]}\H^\al f\|_{L^1(0,\infty)}\lesssim  \| \H^\al f\|_{L^p(0,\infty)}.
$$
Setting $p=2$ gives $\|P^\al_c f\|_{L^\infty(0,\infty)}\lesssim \|  f\|_{L^2(0,\infty)}$ by Plancherel's theorem,  and $p=\infty$ yields $\|P^\al_c f\|_{L^\infty(0,\infty)}\lesssim \|  f\|_{L^1(0,\infty)}$. It thus follows by the Riesz-Thorin theorem  that $\|P^\al_c f\|_{L^\infty(0,\infty)}\lesssim \|  f\|_{L^p(0,\infty)}$, for every 
 $1\leq p\leq2$. 
Moreover, as $\|P^\al_c f\|_{L^p(0,\infty)}\lesssim \|  f\|_{L^p(0,\infty)}$, $1<p<\infty$, 
 it follows again by the Riesz-Thorin theorem that 
\begin{equation}\label{eq:RT1}
\|P^\al_c f\|_{L^q(0,\infty)}\lesssim \|  f\|_{L^p(0,\infty)},\quad 1<p\leq 2,\quad p\leq q<\infty.
\end{equation} 
Since the adjoint of $P_c^\al:L^p(0,\infty)\rightarrow L^q(0,\infty)$ is given by $P_c^\al:L^{q'}(0,\infty)\rightarrow L^{p'}(0,\infty)$, we conclude that also 
\begin{equation}\label{eq:RT2}\|P^\al_c f\|_{L^{p'}(0,\infty)}\lesssim \|  f\|_{L^{q'}(0,\infty)},\quad  1<q'\leq p',\quad 2\leq p'<\infty.
\end{equation}
Combining \eqref{eq:RT1} and \eqref{eq:RT2} thus yields the continuity of $P_c^\al: L^p(0,\infty)\rightarrow L^q(0,\infty)$  for the whole scale $1<p\leq q <\infty$.
\hfill\qed
\end{proof}

\begin{corollary}
If $1<p\leq q <\infty$, then  $B_{\alpha,c}^p\hookrightarrow  B_{\al,c}^q$ with the inclusion map being continuous and dense. Moreover, if $1\leq p\leq \infty$, then $B_{\alpha,c}^1\subset  B_{\al,c}^p$.
\end{corollary}

\begin{proof}
If  $f\in B_{\al,c}^p$ one has $P_c^\al f=f$ and consequently $\| f\|_{L^q(0,\infty)}\lesssim \|  f\|_{L^p(0,\infty)}$ by Theorem~\ref{thm:projection-cont}, which shows that $B_{\al,c}^p$ is continuously embedded into  $B_{\al,c}^q$. Furthermore, if $f\in B_{\al,c}^1$, then $f\in B_{\al,c}^p$ as $f\in L^1(0,\infty)\cap L^\infty(0,\infty)\subset L^p(0,\infty)$.

It remains to show that the inclusion is dense. This however follows from the density properties of $L^p(0,\infty)$ and Theorem~\ref{thm:projection-cont}: Take $f\in B^q_{\al,c}$ and a sequence $(f_k)_{k\geq 0}\subset L^p(0,\infty)\cap L^q(0,\infty)$ such that $f_k\rightarrow f$ in $L^q(0,\infty)$. By the continuity of the projection operator, it follows that $(P_{c}^\al f_k)_{k\geq 0}\subset B^p_{\al,c}\cap B^q_{\al,c}$ and 
$$
\|f-P^\al_c f_k\|_{L^q(0,\infty)}=\|P^\al_c(f- f_k)\|_{L^q(0,\infty)}\lesssim \|f-f_k\|_{L^q(0,\infty)}\rightarrow 0.
$$
\hfill\qed
\end{proof}

\noindent Thus, as in Proposition~\ref{prop:pointwise-ae}, it follows from \cite[Theorem~5.6]{bojaso19} that $\Phi_Nf$ converges almost everywhere to $P_{c}^\al f$ for every $f\in L^p(0,\infty)$ if $ \frac{3}{4}<p<4$.



This leaves us with all ingredients  in place to prove Theorem~\ref{thm:extension2} which is shown along the same line of argument as in the proof of Theorem~\ref{thm:extension}.




\section*{References}

\bibliographystyle{plain}

\begin{thebibliography}{}

\end{thebibliography}


\begin{thebibliography}{10}

\bibitem{abro17}
L.~D. Abreu \& J.~L. Romero.
\newblock {MSE} estimates for multitaper spectral estimation and off-grid
  compressive sensing.
\newblock {\em IEEE Trans. Inform. Theor.}, 63(12):7770--7776, 2017.

\bibitem{baco89}
J.~A. Barcel{\'o} \& A.~C{\'o}rdoba.
\newblock Band-limited functions: ${L}^p$-convergence.
\newblock {\em Trans. Amer. Math. Soc.}, 313:655--669, 1989.

\bibitem{bojaso19}
M.~Boulsane, P.~Jaming, \& A.~Souabni.
\newblock Mean convergence of prolate spheroidal series and their extensions.
\newblock {\em J. Funct. Anal.}, 277(12):108295, 2019.

\bibitem{bow71}
P.~C. Bowie.
\newblock Uncertainty inequalities for the {H}ankel transform.
\newblock {\em {SIAM} J. Math. Anal.}, 2:601--606, 1971.

\bibitem{bow58}
F.~Bowman.
\newblock {\em Introduction to {B}essel {F}unctions}.
\newblock Dover, New York, 1958.

\bibitem{carl66}
L.~Carleson.
\newblock On convergence and growth of partial sums of {F}ourier series.
\newblock {\em Acta Math.}, 116(1):135--157, 1966.

\bibitem{cigupeva99}
O.~Ciaurri, J.~J. Guadalupe, M.~P{\'e}rez \& J.~L. Varona.
\newblock Mean and almost everywhere convergence of {F}ourier-{N}eumann series.
\newblock {\em J. Math. Anal. Appl.}, 236(1):125--147, 1999.

\bibitem{civa09}
O.~Ciaurri \& J.~L. Varona.
\newblock The surprising almost everywhere convergence of Fourier–Neumann series
\newblock {\em J. Comp. Appl. Math.}, 233(3):663--666, 2009.

\bibitem{dy62}
F.~J. Dyson.
\newblock Statistical theory of the energy levels of complex systems.
\newblock {\em J. Math. Phys.}, 3:140--156, 1962.

\bibitem{ghoja11}
S.~Ghobber \& P.~Jaming.
\newblock Strong annihilating pairs for the {F}ourier-{B}essel transform.
\newblock {\em J. Math. Anal. Appl.}, 377(2):501--515, 2011.

\bibitem{humuwhe73}
R.~Hunt, B.~Muckenhoupt \& R.~Wheeden.
\newblock Weighted norm inequalities for the conjugate function and {H}ilbert
  transform.
\newblock {\em Trans. Amer. Math. Soc.}, 1973.

\bibitem{lapo61}
H.~J. Landau \& H.~O. Pollak.
\newblock Prolate spheroidal wave functions, {F}ourier analysis and uncertainty
  {II}.
\newblock {\em Bell Syst. Tech. J.}, 40:65--84, 1961.

\bibitem{lapo62}
H.~J. Landau \& H.~O. Pollak.
\newblock Prolate spheroidal wave functions, {F}ourier analysis and uncertainty
  {III}: {T}he dimension of the space of essentially time- and band limited
  signals.
\newblock {\em Bell Syst. Tech. J.}, 41(4):1295--1336, 1962.

\bibitem{me04}
M.~L. Mehta.
\newblock {\em Random matrices}.
\newblock Pure and Applied Mathematics. Elsevier/Academic Press, Amsterdam,
  2004.

\bibitem{NIST10}
F.~W. Olver, D.~W. Lozier, R.~F. Boisvert \& C.~W. Clark, editors.
\newblock {\em NIST Handbook of {M}athematical {F}unctions}.
\newblock Cambridge University Press, 2010.

\bibitem{petermichl}
S.~Petermichl.
\newblock The sharp bound for the {H}ilbert transform in weighted {L}ebesgue
  spaces in terms of the classical $a_p$ characteristic.
\newblock {\em Amer. J. Math.}, 129:1355--1375, 2007.

\bibitem{po72}
H.~Pollard.
\newblock The convergence almost everywhere of {L}egendre series.
\newblock {\em Proc. Amer. Math. Soc.}, 35(2):442--444, 1972.

\bibitem{rovo99}
M.~R{\"o}sler \& M.~Voit.
\newblock An uncertainty principle for {H}ankel transforms.
\newblock {\em Proc. Amer. Math. Soc.}, 127(1):183--194, 1999.

\bibitem{sle64}
D.~Slepian.
\newblock Prolate spheroidal wave functions, {F}ourier analysis and uncertainty
  {IV}: {M}any dimensions; {G}eneralized prolate spheroidal functions.
\newblock {\em Bell Syst. Tech. J.}, 43:3009--3057, 1964.

\bibitem{Slepian3}
D.~Slepian.
\newblock Some asymptotic expansions for prolate spheroidal wave functions.
\newblock {\em J. Math. Phys.}, 44:99--140, 1965.

\bibitem{sle83}
D.~Slepian.
\newblock Some comments on {F}ourier analysis, uncertainty and modeling.
\newblock {\em {SIAM} Rev.}, 25:379--393, 1983.

\bibitem{slepo61}
D.~Slepian and H.~O. Pollak.
\newblock Prolate spheroidal wave functions, {F}ourier analysis and uncertainty
  {I}.
\newblock {\em Bell Syst. Tech. J.}, 40(1):43--63, 1961.

\bibitem{tho82}
D.~J. Thomson.
\newblock Spectrum estimation and harmonic analysis.
\newblock {\em Proc. IEEE}, 70:1055--1096, 1982.

\bibitem{va94}
J.~L. Varona.
\newblock Fourier series of functions whose {H}ankel transform is supported on
  $[0,1]$.
\newblock {\em Constr. Approx.}, 10:65--75, 1994.

\bibitem{Wa2}
L. Wang.
\newblock A Review of Prolate Spheroidal Wave Functions from the Perspective of Spectral Methods.
\newblock {\em J. Math. Study}, 50:101--142, 2017.

\bibitem{XRY}
H. Xiao, V. Rokhlin \& N. Yarvin
\newblock Prolate spheroidal wavefunctions, quadrature and interpolation.
\newblock {\em Inv. Problems} 17:805--838, 2001.
\end{thebibliography}

\end{document}